\documentclass[reqno,11pt]{amsart}
\usepackage[colorlinks=true, linkcolor=blue, citecolor=blue]{hyperref}

\usepackage{amssymb}
\usepackage{amsmath, graphicx, rotating}
\usepackage{color}
\usepackage{soul}
\usepackage[dvipsnames]{xcolor}

\usepackage[none]{hyphenat}

\allowdisplaybreaks
\usepackage{ifthen}
\usepackage{xkeyval}
\usepackage{todonotes}
\usepackage[T1]{fontenc}
\usepackage{lmodern}
\usepackage[english]{babel} 
\usepackage{microtype}  

\usepackage{ upgreek }
\usepackage{stmaryrd}
\SetSymbolFont{stmry}{bold}{U}{stmry}{m}{n}
\usepackage{amsthm}
\usepackage{float}

\usepackage{ bbm }
\usepackage{ stmaryrd }
\usepackage{ mathrsfs }
\usepackage{ frcursive }
\usepackage{ comment }

\usepackage{pgf, tikz}
\usetikzlibrary{shapes}
\usepackage{varioref}
\usepackage{enumitem}
\usepackage{appendix}

\definecolor{rouge}{rgb}{0.7,0.00,0.00}
\definecolor{vert}{rgb}{0.00,0.5,0.00}
\definecolor{bleu}{rgb}{0.00,0.00,0.8}
\usepackage[margin=1in]{geometry}
\newtheorem{theorem}{Theorem}[section]
\newtheorem*{theorem*}{Theorem}

\newtheorem{proposition}[theorem]{Proposition}

\labelformat{hypothesis}{\textbf{M\kern-0.1mm#1}}

\newtheorem{condition}{Condition}

\newtheorem{assumption}{Assumption}

\labelformat{conditionA}{\textbf{A\kern-0.1mm#1}}

\usepackage{amsthm}

\theoremstyle{definition}
\newtheorem{example}[theorem]{Example}
\newtheorem{remark}[theorem]{Remark}

\numberwithin{equation}{section}

\def\leq{\le}

\def\be{{\beta}}

\def\la{{\lambda}}

\def\cL{{\mathscr L}}

\def\al{{\alpha}}

\def\be{{\beta}}

\def\la{{\lambda}}

\def\al{{\alpha}}
\def\R{{\mathbb{R}}}

\def\va{{\varphi}}
\def\d{{\rm d}}

\begin{document}

\title[Long-Time Behaviors of Branching-Diffusion Processes]
{Long-Time Behaviors of Branching-Diffusion Processes via Spectral Analysis}

\author{Kang Dai \qquad  Jian Wang}

\thanks{\emph{K. Dai:}
School of Mathematics and Statistics, Fujian Normal University, 350117 Fuzhou, P.R. China.
\texttt{kangdaimath@163.com}}
	
\thanks{\emph{J. Wang:}
School of Mathematics and Statistics \& Key Laboratory of Analytical Mathematics and Applications (Ministry of Education) \& Fujian Provincial Key Laboratory of Statistics and Artificial Intelligence, Fujian Normal University, 350117 Fuzhou, P.R. China.
\texttt{jianwang@fjnu.edu.cn}}

\maketitle

\begin{abstract}
We study long-time behaviors for branching-diffusion process corresponding to the drifted Schr\"odinger operator
$\cL = \frac{1}{2} \Delta + \langle \nabla V,\nabla \rangle - K$, where $K$ represents the reduction rate of a population dynamics and $\nabla V$ is a given drift term. In particular, we establish exponential convergence rates for the total mass of this process and characterize its quasi-stationary distribution. The proof is based on a novel transformation in spectral analysis, and heat kernel estimates for  Schr\"odinger operators with unbounded potentials. The result is new even in the one-dimensional setting, which especially improves the recent work \cite{CMS}.

\medskip


\medskip

\end{abstract}
\allowdisplaybreaks

\maketitle
\section{Introduction}\label{Introduction}
Branching-diffusion processes and their applications to population dynamics have been an active subject in probability and mathematical biology. 
In the setting of our paper, we suppose that each individual is described by a trait (or position) $x \in \R^d$, with birth and death occurring in continuous time. More precisely, we study a branching-diffusion process $(Z_t)_{t \ge 0}$ defined as follows: individuals experience trait variations modeled by the following diffusion process on $\R^d$: 
\begin{align}\label{sde}
	\d X_t = \nabla V(X_t)\,\d t + \d B_t,
\end{align}	
where $(B_t)_{t\ge0}$ is a $d$-dimensional standard Brownian motion, and $\nabla V$ is a given drift term such that $V\in C^2(\R^d)$; an individual with trait $x$ gives birth to an individual with the same trait at rate $b(x)$ and dies at rate $d(x)$.
We denote \[K(x) := d(x) - b(x)\] by the reduction rate with trait $x$;
equivalently, $-K(x)$ is the growth rate.
In this paper, we assume that $K(x)$ is locally bounded.
Let $\mathcal{A}_t $ be the set of individuals alive at time $t$. Then, the branching-diffusion process $(Z_t)_{t\ge0}$ is the empirical measure on the trait space $\R^d$, which is defined, for all $t\ge0$, by
\begin{align}\label{em}
	Z_t = \sum_{i \in \mathcal{A}_t} \delta_{X_t^i},
\end{align}
where $X_t^i$ is the trait of individual $i$ at time $t$ and is determined by the diffusion process given by \eqref{sde}.
The total mass of the branching-diffusion process is given by
\[
N_t = \#\mathcal{A}_t = \langle Z_t, 1 \rangle.
\]
The main objective of this note is to study the long-time behavior for $N_t$.

Denote by
\begin{align}\label{L}
	\cL = \frac{1}{2} \Delta + \langle{\nabla V,\nabla}\rangle - K
\end{align}
the generator of the branching-diffusion process $(Z_t)_{t\ge0}$.
It is well known that the operator $\cL$ is symmetric with respect to $\mu(\d x) = e^{2V(x)} \d x$.
The following assumption is always imposed throughout the paper.

\begin{assumption}\label{A0} \hfill
	\begin{itemize}
		\item[{\rm(i)}] \begin{equation}\label{K1}\lim_{|x|\to\infty}\widetilde K(x):= \lim_{|x|\to\infty} \left( K(x) + \frac{1}{2}\Delta V(x)
			+ \frac{1}{2}|\nabla V(x)|^2 \right) = \infty,\end{equation}
		and for any $\theta>0$,
		\begin{equation}\label{K2}\int_{\R^d} e^{-\theta \widetilde K(x)}\,\d x< \infty.\end{equation}
		\item[{\rm(ii)}] \begin{equation}\label{e:ess}\limsup_{|x|\to \infty} \frac{V_-(x)}{|x|^2}=0,\quad \limsup_{|x|\to \infty}\frac{V_-(x)}{\inf_{z\in B(x,|x|/2)}\widetilde K(z)}<\infty;\end{equation} or \begin{equation}\label{e:ess2}\limsup_{|x|\to \infty} \frac{|x|^2}{\inf_{z\in B(x,|x|/2)}\widetilde K(z)}<\infty, \quad \limsup_{|x|\to\infty}\frac{V_-(x)}{|x|\inf_{z\in B(x,|x|/2)}\widetilde K(z)^{1/2}}=0.\end{equation}
	\end{itemize}
\end{assumption}

We will  see that under \eqref{K1}, the operator $-\cL$ has discrete spectrum in $L^2(\R^d;\d \mu)$, and there exists an orthonormal basis in $L^2(\R^d; \d\mu)$ of eigenfunctions $\{\va_n\}_{n\ge0}$ associated with the corresponding eigenvalues $\{\la_n\}_{n\ge0}$ so that
\[
\lambda_0 < \lambda_1 \leq \lambda_2 \leq \cdots, \quad \lim_{n \to \infty} \lambda_n = \infty.
\]
Moreover, there exists a version of $\va_0$ which is continuous and strictly positive, e.g. see the proof of \cite[Proposition 1.2]{CW0}.

For $c, c_0>0$, let $H_{c,c_0}(x)$ be a positive, bounded and continuous function defined on $\R^d$ so that for $x\in \R^d$ with $|x|$ large enough
\begin{equation}\label{e:function}
	H_{c,c_0}(x)=\begin{cases}
		\displaystyle\exp\left(-V_+(x)\right)\left[\exp\left(-c\inf_{z\in B(x,|x|/2)}\widetilde K(z)\right)+\exp(-c_0|x|^2)\right],&\quad \hbox{ if } \eqref{e:ess} \hbox{ holds;} \\
		\displaystyle \exp\left(-V_+(x)\right)\exp\left(-c_0|x|\inf_{z\in B(x,|x|/2)}\widetilde K(z)^{1/2}\right),&\quad \hbox{ if } \eqref{e:ess2} \hbox{ holds.}
	\end{cases}
\end{equation}
The main contribution of this note is as follows.

\begin{theorem}\label{main result0}
	Under Assumption {\rm\ref{A0}}, if
	\begin{align}\label{ass}
		\mu( H_{c,c_0}) <\infty
	\end{align} holds for all $c>0$ large enough and $c_0>0$ small enough, then for any $x\in \R^d$,
	\begin{align}
		\lim_{t \to \infty} e^{\lambda_0 t} \mathbb{E}_x (N_t)
		= \va_0(x) \mu(\va_0);
	\end{align}
	in particular, for any $x \in \R^d$,
	\begin{align}\label{main0}
		\lim_{t \to \infty} \frac{\log (\mathbb{E}_x (N_t))}{t}  = -\lambda_0.
	\end{align}
\end{theorem}

According to \eqref{main0}, the sign of $\la_0$ determines the long-term trend of the total mass $N_t$.
Specifically, if $\la_0 < 0$, then the population grows exponentially; if $\la_0 > 0$, then the population decays exponentially, leading to extinction in the long run; and if $\la_0 = 0$, then the population remains relatively stable.

There have been plenty of results on the study of population dynamics.
Some authors have proved the law of large numbers for the empirical measure $(Z_t)_{t\ge0}$ defined by \eqref{em} by mean of an auxiliary process, see \cite{BDMT,C,GB}.
The existence, the uniqueness and the domain of attraction of quasi-stationary distributions for some diffusion models arising from population dynamics were studied in \cite{CCLMMJ,CMS0,MV,SE}. Furthermore, under the so-called come down from infinity property, the associated $Q$-process is characterized, and populations starting from arbitrarily large sizes enter finite states in finite time, ensuring that the quasi-stationary  distribution attracts all initial distributions in these biological settings.
The readers can be referred to \cite{BT,D,DM,E,FM,M} and references therein for more of the history on the study of population dynamics.

An important source of inspiration for our work  is a recent paper \cite{CMS}, which studies the long-time behavior of linear functionals of branching-diffusion processes via the spectral properties of the Feynman-Kac semigroup in the one-dimensional setting, under the assumptions that $|V(x)|$ has at most linear growth and $|K(x)|$ has at least linear growth.
The example below indicates that our main result, Theorem \ref{main result0}, improves \cite[Theorem 2.1]{CMS} under more general hypotheses, and also extends it to high-dimensional settings.

\begin{example}\label{example} \it Suppose that there exist constants $\alpha\ge0$, $\beta>0$ and $c_1,c_2,c_3,c_4>0$ so that
	\begin{itemize}
		\item[{\rm(i)}] \[\alpha<2,\,\,\alpha\le \beta\] or
		\[\beta\ge2,\,\, \alpha<1+\beta/2.\]
		\item[{\rm(ii)}]  for all $x\in\R^d$, \[\max\big\{V(x) - V_-(x), V_-(x)\big\} \le c_1{|x|}^{\al}+c_2\] and \[K(x)+ \frac{1}{2}\Delta V(x)
		+ \frac{1}{2}|\nabla V(x)|^2 \ge c_3{|x|}^{\be}-c_4.\]
	\end{itemize}
	Then, Theorem {\rm \ref{main result0}} holds.
	In particular, if {\rm(ii)} above holds with $0\le\alpha\le \beta$, $\beta>0$ and $\alpha<2$, then \eqref{ass} is fulfilled for $c>0$ large enough.
	Consequently, Theorem {\rm\ref{main result0}} covers  {\rm\cite[Theorem 2.1]{CMS}}, and it is also more general in sense that it works for the high-dimensional settings.\end{example}

Example \ref{example} is interesting, and it can be directly applied to the Ornstein-Uhlenbeck (O-U) branching-diffusion process (e.g.\ see \cite{EZ}). That is, the diffusion process \eqref{sde} is given by
\begin{equation}\label{SDE-OU}
	\d X_t = c X_t \,\d t + \d B_t,
\end{equation}
where $c\in \R$ is a non-zero constant. The process \eqref{SDE-OU} corresponds to $V(x) = {c}|x|^2 /2$ in our paper. We shall note that the framework of \cite{CMS} excludes the O-U branching-diffusion process, since it requires the term $|V(x)|$ to have at most linear growth. In particular, for \eqref{SDE-OU}, if the reduction rate  \begin{equation}  K(x)
	\ge c_1 |x|^\beta - c_2  \end{equation}
with $c_1, c_2 > 0$ and $\beta > 2$, then the assertion Theorem {\rm \ref{main result0}} holds.
 We point out that a more recent preprint \cite[Remark,\ p.8]{CMS25} points out on the possibility of relaxing such kind linear growth condition in one-dimensional case. Nevertheless, our work presents a rigorous proof for this relaxation even in the high-dimensional setting.

The novelty of our approach is by introducing a new transformation of the generator $\cL$. More explicitly, the importance of this transformation lies in its capacity to absorb the original drift term of $\cL$ into the new potential $\widetilde{K}(x)$, thereby converting the operator $\cL$ into a canonical Schr\"odinger operator $\widetilde\cL$ without drift. With aid of this, we then can make full use of heat kernel estimates for Schr\"odinger operators with unbounded potentials developed in \cite{CW}. The idea of our approach is completely different from that of \cite[Theorem 2.1]{CMS}, where the probabilistic ideas such as the Girsanov transform are adopted. Roughly speaking, the proof of Theorem \ref{main result0} mainly arises from the analytical perspective.
Compared with the Girsanov transform used in \cite[Theorem 2.1]{CMS}, our method here does not require any martingale-related hypotheses (for instance, Novikov's condition), which extensively improves \cite[Theorem 2.1]{CMS}.

As an application of our approach, we can characterize the existence and uniqueness of the quasi-stationary distribution for the branching-diffusion system.
\begin{theorem}\label{THquasi}
	Under Assumption {\rm\ref{A0}} and \eqref{ass},  the probability measure $\nu$ defined by
	\begin{align}\label{quasi-mes}
		\nu(\d x) = \frac{\va_0(x)}{\displaystyle\int_{\R^d} \va_0(y) \,\mu(\d y)} \,\mu(\d x)
	\end{align}
	is the unique quasi-stationary distribution for the Feynman-Kac semigroup $(P_t)_{t\ge 0}$ associated with the branching-diffusion process $(Z_t)_{t\ge 0}$ $($see \eqref{E:FK} below$)$. 
 That is, for any $t>0$ and any bounded measurable function $\phi$, it holds that
	\begin{align}\label{quasi}
		\int_{\R^d} \mathbb{E}_x \left[ \exp\left(-\int_0^t K(X_s) \,\d s\right) \phi(X_t) \right] \nu(\d x) = e^{-\lambda_0 t} \int_{\R^d} \phi(x)\, \nu(\d x).\end{align}
	In particular,
	\begin{align}\label{quasi1}
		\frac{\displaystyle\int_{\R^d} \mathbb{E}_x \left[ \exp\left(-\int_0^t K(X_s)\, \d s\right) \phi(X_t) \right] \,\nu(\d x)}{\displaystyle\int_{\R^d} \mathbb{E}_x \left[ \exp\left(-\int_0^t K(X_s) \,\d s\right)\right] \,\nu(\d x)} = \int_{\R^d} \phi(x) \,\nu(\d x).\end{align}
\end{theorem}

\section{Preliminaries} \label{pf1}

This section is devoted to presenting some preliminaries for the proof of Theorem \ref{main result0}, and it contains two parts. In Subsection \ref{ode}, we introduce a key transformation, which reduces the generator $\cL$ into a Schr\"odinger operator $\widetilde \cL$ without drift, but keeps the associated spectral unchanged.
In Subsection \ref{hk}, we apply heat kernel estimates for Schr\"odinger operators with unbounded potentials, and get explicit decay properties of the eigenfunctions associated with the operator $\cL$.
For simplicity, in what follows, for any $p\in [1,\infty]$, $L^p(\R^d; \d x)$ and $L^p(\R^d; \d\mu)$ are denoted by $L^p(\d x)$ and $L^p(\mu)$ respectively.
We will use $c$ and $C$, with or without subscripts, to denote strictly positive finite constants whose values are insignificant and may change from line to line.

\subsection{A key transformation}\label{ode}
Let $C_{c}^{2}(\R^d)$ be the class of twice differentiable functions on $\R^d$ with compact support. Let $\cL$ be defined by \eqref{L}.
For $f\in C_c^{2}(\R^d)$, define
\begin{equation}\label{e:ope-}\begin{split}
		\widetilde\cL f:&=e^{V} \cL (e^{-V}f)\\
		&=\frac{1}{2} \Delta f
		- \left( K + \frac{1}{2} \Delta V
		+ \frac{1}{2} |\nabla V|^2 \right) f  \\
		&=:\frac{1}{2} \Delta f
		- \widetilde K f.
\end{split}\end{equation} Indeed,
\begin{align*}
	\cL ( e^{-V}f )
	&=\frac{1}{2} \Delta (e^{-V}f)
	+ \langle \nabla V,\nabla(e^{-V}f) \rangle
	- K (e^{-V}f)\\
	&= -\frac{1}{2}e^{-V}f \Delta V
	+\frac{1}{2}e^{-V}f \left| \nabla V \right|^2 - e^{-V} \langle \nabla V,\nabla f \rangle + \frac{1}{2}e^{-V} \Delta f \\
	&\quad~ - e^{-V} f \left| \nabla V \right|^2 + e^{-V} \langle \nabla V,\nabla f \rangle - K (e^{-V}f)\\
	& =e^{-V}\left[\frac{1}{2} \Delta f
	- \left( K + \frac{1}{2} \Delta V
	+ \frac{1}{2} |\nabla V|^2 \right) f \right].
\end{align*}
The advantage of the operator $\widetilde \cL$ is that it is the canonical Schr\"odinger operator with the potential $\widetilde{K}$ (without drift) on $L^2(\d x)$.
It directly follows from \eqref{e:ope-} that $\cL f(x)= \lambda f(x)$ holds for $f\in C^2(\R^d)$ and $\lambda\in \R$, if and only if $\widetilde \cL \tilde f(x)=\lambda \tilde f(x)$, where $\tilde f(x)= e^{V(x)}f(x)$. In particular, the operators $\cL$ and $\widetilde\cL$ have the same eigenvalues, and the associated eigenfunctions differ by a factor $e^{V(x)}$.

Next, we assume that \eqref{K1} holds. Then, it is well known (e.g. see \cite[Chapter 2, Theorem 3.1, p.\ 57]{BS}) that the operator $-\widetilde\cL$ has discrete spectrum in $L^2(\d x)$, and there exists an orthonormal basis in $L^2(\d x)$ of $C^{2}$-eigenfunctions $\{\widetilde\va_n\}_{n\ge0}$ associated with the corresponding eigenvalues $\{\la_n\}_{n\ge0}$ so that
\[
\lambda_0 < \lambda_1 \leq \lambda_2 \leq \cdots, \quad \lim_{n \to \infty} \lambda_n = \infty.
\]
Hence, $-\cL$ also has discrete spectrum consisting of eigenvalues $\{\la_n\}_{n\ge0}$. Moreover, let $\{\va_n\}_{n\ge0}$ be the associated eigenfunctions of $-\cL$. Then, $\va_n(x) = e^{-V(x)} \widetilde\va_n(x)$. In particular, $\| \va_n \|_{L^2(\mu)} = \| \widetilde\va_n \|_{L^2(\d x)} = 1$. The first eigenfunction $\va_0$ is called ground state in the literature.

\subsection{Estimates for eigenfunctions $\{\va_n\}_{n\ge1}$}\label{hk} Assume that Assumption  \ref{A0}(i) holds.
Let $(\widetilde P_t)_{t\ge0}$ be the Schr\"odinger semigroup associated with the operator $\widetilde \cL$, and $\widetilde p(t,x,y)$ be the associated heat kernel. That is,
for any $f\in L^2(\d x)$,
\[\widetilde P_t f(x)=\int_{\R^d} \widetilde p(t,x,y) f(y)\,\d y,\quad t>0,~x\in \R^d.\] Note that the existence of heat kernel $\widetilde p(t,x,y)$ is
well known, see \cite[Chapter 3]{CZ} or \cite[Theorem B.7.1]{S}.
The key point to get estimates for the eigenfunctions $\{\va_n\}_{n\ge1}$ is to obtain an upper bound for $\widetilde p(t,x,x)$. We first suppose that $\inf_{x\in \R^d}\widetilde K(x)>0.$ Then, for any $t>s>0$ and $x\in \R^d$,
\begin{align*}
	&\widetilde p(t,x,x)\\
	&=\int_{\R^d} \widetilde p(t-s,z,x)\widetilde p(s,x,z)\,\d z\\
	&\le c_1(t-s)^{-d/2} \widetilde P_{s}1(x)\\
	&\le c_2 (t-s)^{-d/2} \left[\exp\left(-c_3 s\left(1+\inf_{z\in B(x,|x|/2)}\widetilde K(z)\right)\right)+\exp\left(-c_3\left(s+\frac{(1+|x|)^2}{s}\right)\right)\right],
	\end{align*} where the equality above follows from the semigroup of the heat kernel $\widetilde p(t,x,y)$, in the first inequality we used the fact that $\widetilde p(t-s,z,x)\le c_1 (t-s)^{-d/2}$ for all $x,z\in \R^d$ by $\inf_{x\in \R^d}\widetilde K(x)\ge 0$ and the Feynman-Kac formula for the Schr\"odinger semigroup $(\widetilde P_t)_{t\ge0}$ (e.g.\ see \cite[(3)]{CW}), and the last inequality can be deduced from the proof of \cite[Lemma 2.5]{CW}. (Note that in our setting we do not require that the potential $\widetilde K(x)$ satisfies \cite[Assumption (H)]{CW}, so the term $\inf_{z\in B(x,|x|/2)}\widetilde K(z)$ appears. Note also that though \cite[Lemma 2.5]{CW} only claims the assertion for all $x\in \R^d$ with $|x|\ge2$,  it still holds for all $x\in \R^d$ by choosing $c_3$ small enough thanks to \cite[Lemma 2.6]{CW}.)
It is clear that, under \eqref{K2}, for all $t>0$,
\[\int_{\R^d} \widetilde p(t,x,x)\,\d x < \infty. \]
Then, for all $t>0$ and $x,y\in \R^d$,
\begin{equation}\label{e:add}\widetilde p(t,x,y)=\sum_{n=0}^\infty e^{-\lambda _n t} \widetilde\va_n(x) \widetilde\va_n(y);\end{equation}
e.g.\ see \cite[Lemma 2.1]{DS}.
Therefore, for any $n\ge0$, $t>0$ and $x\in \R^d$,
\begin{align*}|\widetilde \va_n(x)| &=e^{\lambda_nt/2} \left(e^{-\lambda_n t} |\widetilde\va_n(x)|^2\right)^{1/2}\\
	&\le e^{\lambda_nt/2}\left(\sum_{n=0}^\infty e^{-\lambda _n t} |\widetilde\va_n(x)|^2\right)^{1/2}=e^{\lambda_nt/2} \widetilde p(t,x,x)^{1/2}. \end{align*}
Hence, for any $n\ge0$, $t>s>0$ and $x\in \R^d$,
\begin{align*}
	|\widetilde \va_n(x)| \le c_4 (t-s)^{-d/4}e^{\lambda_nt/2} \left[\exp\left(-c_5 s\left(1+\inf_{z\in B(x,|x|/2)}\widetilde K(z)\right)\right)+\exp\left(-c_5\left(s+\frac{(1+|x|)^2}{s}\right)\right)\right]\!,
\end{align*}
which in turn implies that for any $n\ge0$, $t>s>0$ and $x\in \R^d$,
\begin{equation}\label{e:eif}
	\begin{split}|\va_n(x)|\le
		&c_4 (t-s)^{-d/4}e^{\lambda_nt/2}e^{-V(x)}\\
		&\times\left[\exp\left(-c_5 s\left(1+\inf_{z\in B(x,|x|/2)}\widetilde K(z)\right)\right)+\exp\left(-c_5\left(s+\frac{(1+|x|)^2}{s}\right)\right)\right].\end{split}
\end{equation}
Note that, since we assume that  $\inf_{x\in \R^d}\widetilde K(x)>0$, the term $\left(1+\inf_{z\in B(x,|x|/2)}\widetilde K(z)\right)$ in \eqref{e:eif} can be replaced with $\inf_{z\in B(x,|x|/2)}\widetilde K(z)$ by adjusting the constant $c_5$ properly. We write the form \eqref{e:eif} since it holds for general cases, which will be explained later.

For general case, we set $\widehat K(x)=\widetilde K(x) + m$, where $m>0$ satisfies $\inf_{x\in\R^d}\widetilde K(x)+m>0$. The existence of the constant $m$ is ensured by $V\in C^2(\R^d)$ and \eqref{K1}. Consider
\[
\widehat\cL
=\widetilde\cL - m
=\frac{1}{2} \Delta - (\widetilde K(x) + m)
=\frac{1}{2} \Delta - \widehat K(x).
\]
Denote by $\widehat p(t,x,y)$ the heat kernel associated with $\widehat\cL$. It holds that
\[
\widehat p(t,x,y) = e^{-mt}\widetilde p(t,x,y),\quad t>0,~x,y\in \mathbb{R}^d.
\] On the other hand, it is easy to see that
$
- \widetilde \cL f=\la  f
$ holds with $f\in C^2(\R^d)$ and $\lambda\in \R$, if and only if,
$
-\widehat\cL f= (\la+ m) f$.
This implies that $\widehat\cL$ and $\widetilde\cL$ have the same eigenfunctions $\{\widetilde\va_n\}_{n\ge0}$, and $\widehat{\la}_n=\la_n + m$,
where $(\widehat{\la}_n)_{n\ge0}$ (resp.\ $(\la_n)_{n\ge0}$) are the associated eigenvalues for the operator $-\widehat\cL$ (resp. $-\widetilde\cL$).
With these two facts at hand, we can apply the argument above to the operator $\widehat\cL$ and get that \eqref{e:eif} still holds true.

Furthermore, we can get the following two conclusions by the arguments above. First, according to \eqref{e:add}, for all $t>0$,
\begin{equation}\label{e:add2}\sum_{n=0}^\infty e^{-\lambda_n t}=\sum_{n=0}^\infty e^{-\lambda_n t}\|\widetilde \va_n\|_{L^2(\d x)}^{2}=\int_{\R^d}\tilde p(t,x,x)\,\d x<\infty.\end{equation} Secondly, by \eqref{e:eif}, for any $T_0>1$, there is a positive and locally bounded function $h(x)$ such that for all $x\in \R^d$ and $n\ge0$,
\[|\va_n(x)|\le c_{6} e^{\lambda_n T_0/2} h(x).\] The proposition below indicates that under Assumption \ref{A0}, we can take $h(x)=H_{c,c_0}(x)$ defined by \eqref{e:function}.

For $a\in \R$, set $a_+:=\max\{a,0\}$ and $a_-:=\max\{-a,0\}$. Then, we have
\begin{proposition} \label{P:2.1} Suppose that Assumption {\rm\ref{A0}} holds. Then the following statements hold.
	\begin{itemize}
		\item[{\rm(i)}] Assume that \begin{equation}\label{e:ess*}\limsup_{|x|\to \infty} \frac{V_-(x)}{|x|^2}=0,\quad \limsup_{|x|\to \infty}\frac{V_-(x)}{\inf_{z\in B(x,|x|/2)}\widetilde K(z)}<\infty.\end{equation}
		Then, for any $c>0$, there are constants $c_1,c_2, r_0>0$ and $T_0\ge1$ such that for all $n\ge0$, $x\in \R^d$ with $|x|\ge r_0$,
		\[ |\va_n(x)|\le c_1e^{\lambda_n T_0/2}\exp\left(-{V_+(x)}\right)\left[\exp\left(-c\inf_{z\in B(x,|x|/2)}\widetilde K(z)\right)+\exp\left(- c_2 |x|^2 \right)\right].\]
		\item[{\rm(ii)}]  Assume that \begin{equation}\label{e:ess2*}\limsup_{|x|\to \infty} \frac{|x|^2}{\inf_{z\in B(x,|x|/2)}\widetilde K(z)}<\infty, \quad \limsup_{|x|\to\infty}\frac{V_-(x)}{|x|\inf_{z\in B(x,|x|/2)}\widetilde K(z)^{1/2}}=0.\end{equation} Then, there are constants $c_1,c_2,r_0>0$ and $T_0\ge1$ such that for all $n\ge0$ and $x\in \R^d$ with $|x|\ge r_0$,
		\[ |\va_n(x)|\le c_1e^{\lambda_n T_0/2}\exp\left(-{V_+(x)}\right)\exp\left(-c_2|x|\inf_{z\in B(x,|x|/2)}\widetilde K(z)^{1/2}\right).\]
	\end{itemize}
\end{proposition}
\begin{proof} According to \eqref{K1}, there is a constant $r_1>0$ for all $x\in \R^d$ with $|x|\ge r_1$,
	\[\inf_{z\in B(x,|x|/2)}\widetilde K(z)>0.\]
	
	(1) We first prove the assertion (i). Suppose that \eqref{e:ess*} holds. For any $c>0$, we can choose $T_0\ge1$ and $r_2\ge r_1$ such that  for all $x\in \R^d$ with $|x|\ge r_2$,
	\[\frac{c_5T_0}{2}\inf_{z\in B(x,|x|/2)} \widetilde K(z)\ge \max\left\{ c\inf_{z\in B(x,|x|/2)} \widetilde K(z), V_-(x)\right\}.\] On the other hand, we can further take $r_3\ge r_2$ such that for all $x\in \R^d$ with $|x|\ge r_3$,
	\[\frac{c_5(1+|x|)^2}{2T_0}\ge V_-(x).\] Here, $c_5$ is the constant given in \eqref{e:eif}. Then, the desired assertion holds with $r_0=r_3$  by applying \eqref{e:eif} with $s=t/2$ and $t=T_0$.

	(2) To prove the assertion (ii), we only need to consider the case that $x\in \R^d$ with $|x|\ge r_1$. We take $s=\frac{1+|x|}{\left(1+\inf_{z\in B(x,|x|/2)}\widetilde K(z)\right)^{1/2}}$ in \eqref{e:eif}. Then, under \eqref{e:ess2*}, there are constants $c_1,c_2>0$ and $r_2\ge r_1$ such that for all $n\ge0$, $t>s$ and $x\in \R^d$ with $|x|\ge r_2$,
	\[ |\va_n(x)|\le c_1 (t-s)^{-d/4} e^{\lambda_n t/2}\exp\left(-{V_+(x)}\right)\exp\left(-c_2|x|\inf_{z\in B(x,|x|/2)}\widetilde K(z)^{1/2}\right).\]
	Note that \[\sup_{|x|\ge r_2}\frac{1+|x|}{\left(1+\inf_{z\in B(x,|x|/2)}\widetilde K(z)\right)^{1/2}}<\infty,\] also due to \eqref{e:ess2*}. Hence, the desired assertion follows  by taking \[T_0=t=1+2\sup_{|x|\ge r_2}\frac{1+|x|}{\left(1+\inf_{z\in B(x,|x|/2)}\widetilde K(z)\right)^{1/2}}>0\] and $r_0=r_2$. The proof is complete.
\end{proof}

\section{General result and proofs of results in Section \ref{Introduction}}\label{PF1}

In this section, we will establish a general result for long-time behaviors of branching-diffusion process $(Z_t)_{t\ge0}$, which immediately yields Theorem \ref{main result0}.
For this, we will consider long-time behaviors of the linear functional $\langle Z_t, \phi \rangle $ for suitable functions $\phi$, which in turn is based on the relationship between the branching-diffusion process $(Z_t)_{t\ge0}$ and the following Feynman-Kac formula for the Schr\"odinger semigroup $(P_t)_{t\ge0}$:
\begin{equation}\label{E:FK}
P_t \phi(x) = \mathbb{E}_{\delta_x} (\langle Z_t, \phi \rangle) = \mathbb{E}_x \left[ \exp \left( \int_0^t -K(X_s) \, \d s \right) \phi(X_t) \right], \quad \phi \in C_b(\R^d), t\ge0, x \in\R^d,
\end{equation}
where $(X_t)_{t\ge0}$ is the diffusion process defined by \eqref{sde}.

Recall that under \eqref{K1}, the generator $-\cL$ of  the Schr\"odinger semigroup $(P_t)_{t\ge0}$ has discrete spectrum consisting of eigenvalues $\{\la_n\}_{n\ge0}$, which correspond to the eigenfunctions $\{\va_n\}_{n\ge0}$.
Let $\Pi$ denote the projection onto the space generated by the first eigenfunction $\va_0$; that is, for any measurable function $g$ such that $g\varphi_0 \in L^1(\mu)$,  
\begin{align}\label{pi}
	\Pi(g)(x) = \va_0(x) \int_{\R^d} g(y)\va_0(y) \,\mu(\d y),\quad x\in \R^d.
\end{align}
Let $H_{c,c_0}(x)$ be a positive, bounded and continuous function defined in \eqref{e:function}; that is, for $x\in \R^d$ with $|x|$ large enough,
\begin{equation*}
	H_{c,c_0}(x)=\begin{cases}\displaystyle
		\exp\left(-V_+(x)\right)\left[\exp\left(-c\inf_{z\in B(x,|x|/2)}\widetilde K(z)\right)+\exp(-c_0|x|^2)\right],&\quad \hbox{ if } \eqref{e:ess} \hbox{ holds;} \\
		\displaystyle \exp\left(-V_+(x)\right)\exp\left(-c_0|x|\inf_{z\in B(x,|x|/2)}\widetilde K(z)^{1/2}\right),&\quad \hbox{ if } \eqref{e:ess2} \hbox{ holds.}
	\end{cases}
\end{equation*}
In particular, according to Proposition \ref{P:2.1}, under Assumption \ref{A0}, there is a constant $C_0>0$ such that for all $x\in \R^d$ and $n\ge0$,
\begin{equation}\label{e:add3}|\va_n(x)| \le C_0  e^{\lambda_n T_0/2} H_{c,c_0}(x)\end{equation} for some proper $T_0\ge1$ and $c,c_0>0$. Here, we should mention that when \eqref{e:ess} holds, one can take $c>0$ large enough in the definition of $H_{c,c_0}(x)$.

\begin{theorem}\label{main result1}
	Under Assumption {\rm\ref{A0}}, for any measurable function $\phi$ on $\R^d$ satisfying
	\begin{align}\label{ass1}
		\int_{\R^d} H_{c,c_0}(y) |\phi(y)| \,\mu(\d y) <\infty
	\end{align} with $c>0$ large enough and $c_0>0$ small enough,
	there are constants $T_0>0$ and $C_0>0$ such that all $t > T_0$,
	\begin{align}\label{results1}
		\|e^{\lambda_0t}P_t(\phi)-\Pi(\phi)\|_{L^{\infty}(\mu)}
		\le C_0 e^{-(\lambda_1-\lambda_0)t}.
	\end{align}
\end{theorem}

As an application of Theorem \ref{main result1}, we can give the

\begin{proof}[Proof of Theorem $\ref{main result0}$] Suppose that \eqref{ass} holds, i.e., \eqref{ass1} is satisfied with $\phi\equiv1$. Then, according to Theorem \ref{main result1}, we have the following asymptotic for the total mass: for any $x\in \R^d$,
	\[
	\lim_{t \to \infty} e^{\lambda_0 t} \mathbb{E}_x (N_t) = \lim_{t \to \infty} e^{\lambda_0 t} P_t1(x)=\Pi(1)(x)
	= \va_0(x) \int_{\R^d} \va_0(y) \,\mu(\d y).
	\] The proof is complete.	
\end{proof}

Next, we will present the
\begin{proof}[Proof of Theorem {\rm\ref{main result1}}]
	Assume that Assumption \ref{A0} holds. Recall that $\widetilde\cL =e^{V} \cL (e^{-V} \cdot)$ and $\displaystyle \int_{\R^d} \widetilde p(t,x,x)\,\d x < \infty$ for any $t>0$.
	Then, it holds that
	\begin{equation}\label{relat}\widetilde P_t f = e^V P_t(e^{-V}f),\quad  f\in C^\infty_c(\R^d),\,t>0. \end{equation}
	Thus, for any $f\in C_c^\infty(\R^d)$, $t>0$ and $x\in \R^d$,
	\begin{align*}
		P_t f(x) &= e^{-V(x)} \widetilde P_t(e^{V(\cdot)}f(\cdot))(x)\\
		&= e^{-V(x)} \int_{\R^d} \widetilde p(t,x,y) f(y) e^{V(y)}\,\d y \\
		&= \int_{\R^d} \widetilde p(t,x,y)e^{-V(x)}e^{-V(y)} f(y) \,\mu(\d y),
	\end{align*}
	which implies the heat kernel associated with $(P_t)_{t\ge0}$ exists (with respect to the measure $\mu$), denoted by $p(t,x,y)$. That is, for all $t>0$ and $x,y\in \R^d$,  \[p(t,x,y)=\widetilde p(t,x,y)e^{-V(x)}e^{-V(y)}.\]
	Furthermore, for all $t>0$,
	\[\int_{\R^d}  p(t,x,x)\,\mu(\d x) = \int_{\R^d} \widetilde p(t,x,x)e^{-2V(x)}\,\mu(\d x) = \int_{\R^d} \widetilde p(t,x,x)\,\d x < \infty.\]
	Therefore, with aid of this and the remark in the end of Subsection \ref{ode}, we have
	\begin{align}\label{sd}
		p(t,x,y)=\sum_{n=0}^\infty e^{-\lambda _n t} \va_n(x) \va_n(y),\quad t>0,~x,y\in \R^d,\end{align}
	where $\{\la_n\}_{n\ge0}$ are the associated eigenvalues corresponding to the eigenfunctions $\{\va_n\}_{n\ge0}$.
	
	For any measurable function $\phi$, any $t>0$ and $x\in \R^d$,
	\begin{align*}
		P_t(\phi-\Pi(\phi))(x)
		=&\sum_{n=0}^{\infty} e^{-\la_{n}t}\va_n(x)
		\int_{\R^d}\va_n(y)\left[\phi(y)-\Pi(\phi)(y)\right]\,\mu(\d y)\\
		=&e^{-\la_{0}t}\va_0(x)
		\int_{\R^d} \va_0(y)\left[\phi(y)-\Pi(\phi)(y)\right]\,\mu(\d y)\\
		&+\sum_{n=1}^{\infty} e^{-\la_{n}t}\va_n(x)
		\int_{\R^d} \va_n(y)\left[\phi(y)-\Pi(\phi)(y)\right]\,\mu(\d y)\\
		=&\sum_{n=1}^{\infty} e^{-\la_{n}t}\va_n(x)
		\int_{\R^d} \va_n(y)\phi(y)\,\mu(\d y),
	\end{align*} where we used the definition of $\Pi(\phi)(x)$ and the facts that $\displaystyle\int_{\R^d} \va_n(x)\va_0(x)\,\mu(\d x)=0$ for all $n\ge1$ and $\|\va_n\|_{L^2(\mu)}=1$.
	Hence, for any $t> T_0$,
	\begin{align*}
		&\|P_t(\phi-\Pi(\phi))\|_{L^{\infty}(\mu)}\\
		&= \left\| \sum_{n=1}^{\infty} e^{-\la_{n}t}\va_n(\cdot)
		\int_{\R^d} \va_n(y)\phi(y)\,\mu(\d y) \right\|_{L^{\infty}(\mu)} \\
		&\le \sum_{n=1}^{\infty} e^{-\la_{n}t}
		\|\va_n\|_{L^{\infty}(\mu)}
		\int_{\R^d} |\va_n(y)\phi(y)| \,\mu(\d y)\\
		&\le c_1 \sum_{n=1}^{\infty} e^{-\la_{n}t} e^{\la_n T_0/2}
		\int_{\R^d} e^{\la_n T_0/2} H_{c,c_0}(y) |\phi(y)| \,\mu(\d y) \\
		&= c_1 \sum_{n=1}^{\infty} e^{-\la_{n}t} e^{\la_n T_0}
		\int_{\R^d} H_{c,c_0}(y) |\phi(y)| \,\mu(\d y)\le c_2 \sum_{n=1}^{\infty} e^{-\la_{n}(t-T_0)},
	\end{align*}
	where in both inequalities above we used \eqref{e:add3} and \eqref{ass1}.
	This, along with the fact that for all $t\ge T_0+1$,
	\[ \sum_{n=1}^{\infty} e^{-\la_{n}(t-T_0)}= e^{-\la_{1}t} e^{\la_{1} T_0} \sum_{n=1}^{\infty} e^{-(\la_{n}-\la_{1})(t-T_0)}=:c_3 e^{-\la_{1}t},\] (thanks to \eqref{e:add2}), yields that
	\[
	\|P_t(\phi-\Pi(\phi))\|_{L^{\infty}(\mu)} \le c_4 e^{-\la_1 t}, \quad t\ge T_0+1.
	\]
	
	Furthermore, according to the fact that $e^{\lambda_0 t}P_t \va_0(x)=\va_0(x)$ for all $t>0$ and $x\in \R^d$, we have
	\[
	\|e^{\lambda_0 t}P_t(\phi)-\Pi(\phi)\|_{L^{\infty}(\mu)}= e^{\lambda_0 t}\|P_t(\phi-\Pi(\phi))\|_{L^{\infty}(\mu)}.\]
	Therefore, the proof is complete.
\end{proof}

\begin{remark}
	We add a comment on the proof of Theorem \ref{main result1}. 		
	We start from the following computation.
	Assume that \eqref{K1} holds. For any measurable function $\phi$, $t>0$ and $x\in \R^d$, it holds that
	\begin{align*}
		\|P_t(\phi-\Pi(\phi))\|_{L^{\infty}(\mu)}&\le \sum_{n=1}^{\infty} e^{-\la_{n}t} \|\va_n\|_{L^{\infty}(\mu)}
		\left|\int_{\R^d} \va_n(y)\phi(y)\,\mu(\d y)\right|\\
		&\le \sum_{n=1}^{\infty} e^{-\la_{n}t} \|\va_n\|^2_{L^{\infty}(\mu)}\|\phi\|_{L^1(\mu)}\end{align*} and
	\begin{align*}
		\|P_t(\phi-\Pi(\phi))\|_{L^{\infty}(\mu)}&\le \sum_{n=1}^{\infty} e^{-\la_{n}t} \|\va_n\|_{L^{\infty}(\mu)}
		\left|\int_{\R^d} \va_n(y)\phi(y)\,\mu(\d y)\right|\\
		&\le \sum_{n=1}^{\infty} e^{-\la_{n}t}\|\va_n\|_{L^{\infty}(\mu)}\|\va_n\|_{L^{2}(\mu)}\|\phi\|_{L^{2}(\mu)}\\
		&=\sum_{n=1}^{\infty} e^{-\la_{n}t}\|\va_n\|_{L^{\infty}(\mu)}\|\phi\|_{L^{2}(\mu)}.\end{align*}
	Then, if there is $t_0>0$ so that $\|P_{t_0}\|_{L^{2}(\mu) \to L^{\infty}(\mu)}<\infty$, it follows from  $P_t \va_n(x) = e^{-\la_n t}\va_n(x)$ for all $t>0$ and $x\in \R^d$ that
	\begin{align*}\|\va_n\|_{L^{\infty}(\mu)}
		&\le e^{\la_n t_0} \|P_{t_0} \va_n\|_{L^{\infty}(\mu)}\\ &\le e^{\la_n t_0} \|P_{t_0}\|_{L^{2}(\mu) \to L^{\infty}(\mu)} \|\va_n\|_{L^{2}(\mu)}\\  &\le  e^{\la_n t_0}\|P_{t_0}\|_{L^{2}(\mu) \to L^{\infty}(\mu)}.\end{align*}
	Hence,  for any $\phi\in L^2(\mu)\cup L^1(\mu)$,  there exist $T_0>0$ and $c_1>0$ such that for all $t\ge T_0$,
	\[
	\|P_t(\phi-\Pi(\phi))\|_{L^{\infty}(\mu)} \le c_1 e^{-\la_1 t};\] that is,
	\[ \|e^{\lambda_0t}P_t(\phi)-\Pi(\phi)\|_{L^{\infty}(\mu)}
	\le c_1 e^{-(\lambda_1-\lambda_0)t}.\]
	However, it is non-trivial to show that $\|P_{t_0}\|_{L^{2}(\mu) \to L^{\infty}(\mu)}<\infty$ for some $t_0>0$, which is related to the so-called ultracontractivity of the semigroup $(P_t)_{t\ge0}$. Moreover, the arguments above also require that $\phi\in L^2(\mu)\cup L^1(\mu)$. In particular, taking $\phi \equiv 1$ yields the assertion of Theorem \ref{main result0}. Nevertheless, this
	implies that $\mu$ is a finite measure, which is a strong restriction to study long-time behaviors of the total mass $N_t$. Indeed, the key ingredient to obtain  Theorem \ref{main result1} is to establish explicit decay properties of the eigenfunctions $\{\va_n\}_{n\ge1}$.
\end{remark}

Next, we provide the

\begin{proof}[Proof of Example $\ref{example}$] It is clear that Assumption \ref{A0}(i) holds. Next, we verify that Assumption \ref{A0}(ii) and \eqref{ass} are satisfied, case by case.
	
	(i) When $\alpha\le \beta$ and $\alpha<2$, \eqref{e:ess} in Assumption \ref{A0}(ii) holds. Recall that $H_{c,c_0}(x)$ is positive, bounded and continuous. Below, we take $r_0\ge1$ large enough. Then,
	\begin{align*}
		\mu(H_{c,c_0}) & \le c_1 + \displaystyle\int_{\{|x|\ge r_0\}} e^{-V_+(x)}\left[\exp\left(-c\inf_{z\in B(x,|x|/2)} \widetilde K(z)\right) + \exp{\left(-c_0|x|^2\right)}\right] e^{2V(x)}\,\d x\\
		&\le c_1 + c_2\displaystyle\int_{\{|x|\ge r_0\}} \left[\exp{\left(-cc_3|x|^{\be}\right)}+\exp{\left(-c_0|x|^2\right)}\right] e^{V_{+}(x)-2V_{-}(x)}\,\d x \\
		&\le c_1 + c_4\displaystyle\int_{\{|x|\ge r_0\}} \left[\exp{\left(-cc_3|x|^{\be}+c_5 |x|^{\al}\right)} + \exp{\left(-c_0|x|^2+c_5 |x|^{\al}\right)}\right]\,\d x\\
		&< \infty,
	\end{align*}
	where in the third inequality we used the condition $V(x) - V_-(x)\le c_5{|x|}^{\al}+c_6$, and in the last inequality we used that facts that $c$ can be chosen large enough, $\beta\ge \alpha$ and $\alpha<2$.
	
	(ii) When $\beta\ge2$ and $\alpha<1+\beta/2$,  \eqref{e:ess2} in Assumption \ref{A0}(ii) holds. Below, we take $r_0\ge1$ large enough. Then,
	\begin{align*}
		\mu( H_{c,c_0}) & \le c_1+ \displaystyle\int_{\{|x|\ge r_0\}} \exp\left(-V_+(x)\right)\exp\left(-c_0|x|\inf_{z\in B(x,|x|/2)}\widetilde K(z)^{1/2}\right) e^{2V(x)}\,\d x\\
		&\le c_1 + c_2\displaystyle\int_{\{|x|\ge r_0\}} \exp{\left(-cc_3|x|^{\be/2 +1}\right)} e^{V_{+}(x)-2V_{-}(x)}\,\d x \\
		&\le c_1 + c_4\displaystyle\int_{\{|x|\ge r_0\}} \exp{\left(-cc_3|x|^{\be/2 +1}+c_5 |x|^{\al}\right)} \,\d x < \infty,
	\end{align*} where in the last inequality we used $\alpha<1+\beta/2$. The proof is complete.
\end{proof}

Finally, we give the
\begin{proof}[Proof of Theorem $\ref{THquasi}$] We use the notations in the proof of Theorem \ref{main result1}.
	For any $t>0$ and any bounded measurable function $\phi$,
	\begin{equation}\label{e:proof-add}\begin{split}
			& \int_{\R^d} \mathbb{E}_x \left[ \exp\left(-\int_0^t K(X_s)\,\d s\right) \phi(X_t) \right] \varphi_0(x) \,\mu(\d x) \\
			&= \int_{\R^d} \left[ \int_{\R^d} \phi(y) p(t, x, y) \,\mu(\d y) \right] \varphi_0(x) \,\mu(\d x) \\
			&= \int_{\R^d}\int_{\R^d} \phi(y)\left[ \sum_{n=0}^{\infty} e^{-\lambda_n t} \varphi_n(x) \varphi_n(y) \right] \varphi_0(x)  \,\mu(\d y) \,\mu(\d x) \\
			&= \sum_{n=0}^{\infty} e^{-\lambda_n t}  \int_{\R^d} \varphi_n(y) \phi(y) \,\mu(\d y) \int_{\R^d} \varphi_0(x) \varphi_n(x) \,\mu(\d x)\\
			&= e^{-\lambda_0 t}\int_{\R^d} \phi(y) \varphi_0(y) \,\mu(\d y),
	\end{split}\end{equation}
	where in the second equality we used \eqref{sd}, and the last equality follows from the fact that $\displaystyle\int_{\R^d} \va_n(x)\va_0(x)\,\mu(\d x)=0$ for all $n\ge1$.
	
	Note that, by \eqref{ass} and \eqref{e:add3}, as well as the fact that $\varphi_0(x)$ is continuous and strictly positive on $\R^d$,
	$\displaystyle\int_{\R^d} \varphi_0(y) \,\mu(\d y)<\infty$. In particular, $\nu$ given by \eqref{quasi-mes} is well defined. This along with \eqref{e:proof-add} yields \eqref{quasi}.
	
	Taking $\phi=1$ in \eqref{e:proof-add}, we get $$\int_{\R^d} \mathbb{E}_x \left[ \exp\left(-\int_0^t K(X_s)\,\d s\right) \right] \varphi_0(x) \,\mu(\d x) = e^{-\lambda_0 t}\int_{\R^d}  \varphi_0(y) \,\mu(\d y).$$ Then, \eqref{quasi1} follows from \eqref{quasi}.
	
	The proof of the uniqueness for quasi-stationary distribution is similar to these of  \cite[Theorem 2.3 and Corollary 2.4]{CMS}, so we omit it.
\end{proof}

Indeed, following the arguments in \cite[Section 4]{CMS}, we can further apply Theorem \ref{main result1} to study the domain of attraction of quasi-stationary distribution of the branching-diffusion process $(Z_t)_{t\ge0}$, as well as the existence of the associated $Q$-process; see \cite[Theorems 2.3--2.6]{CMS}. Since in this paper we focus on long-time behaviors of the total mass $N_t$, we do not present the details here.

\vspace{0.3cm}
\noindent\textbf{Acknowledgment}.
The research is supported by the National Key R\&D Program of China (2022YFA1006003) and  the National Natural Science Foundation of China (Nos.\ 12071076 and 12225104).

\

\noindent\textbf{Data Availability}.
Data sharing not applicable to this article as no datasets were generated or analysed during the current study.

\

\noindent\textbf{Declarations}.

\

\noindent\textbf{Competing interests}
The authors declare no competing interests.

\

\noindent\textbf{Conflict of Interest}
The authors have no competing interests to declare that are relevant to the content of this article.

\

\noindent\textbf{Publisher's Note}
Springer Nature remains neutral with regard to jurisdictional claims in published maps and institutional affiliations.

\

\vspace{0.3cm}
\end{document}